\documentclass[11pt]{article}
\usepackage[dvips]{epsfig}
\usepackage{amsmath,amssymb,euscript,graphicx,amsfonts,verbatim}
\usepackage{enumerate,color}
\usepackage[colorlinks,breaklinks,backref]{hyperref}
\usepackage{hyperref}
\usepackage{backref}

\usepackage{xcolor}
\usepackage{mathtools}
\usepackage[all,arc,curve,color,frame,pdf,line]{xy}
\usepackage{caption}
\usepackage{tikz}
\usetikzlibrary{shapes.geometric}

\newtheorem{theorem}{Theorem}[section]
\newtheorem{proposition}[theorem]{Proposition}
\newtheorem{lemma}[theorem]{Lemma}
\newtheorem{corollary}[theorem]{Corollary}

\newtheorem{problem}[theorem]{Problem}

\newtheorem{remark}[theorem]{Remark}

\newenvironment{proof}{{\noindent \sc Proof. }}{\hfill $\Qed$\\}

\newcommand{\Qed}{\rule{2.5mm}{3mm}}
\newcommand{\Aut}{\hbox{{\rm Aut}}\,}

\def\Z{{\mathbb Z}}

\newcommand{\ZZ}{\mathbb{Z}}

\newcounter{case}

\renewcommand{\thecase}{\arabic{case}}

\newcounter{subcase}

\numberwithin{subcase}{case}

\begin{document}


\title
{On cubic vertex-transitive graphs of given girth
}
\author{Ted Dobson and Ademir Hujdurovi\'{c}\\ University of Primorska, Koper, Slovenia\and Wilfried Imrich and Ronald Ortner \\ Montanuniversit\"at Leoben, A-8700 Leoben, Austria
}
\date{}
\maketitle
\begin{abstract}
A set of vertices of a graph is distinguishing if the only automorphism that preserves it is the identity. The minimal size of such sets, if they exist, is the distinguishing cost. The distinguishing costs of
vertex transitive cubic graphs are well known if they are 1-arc-transitive, or if they have two edge orbits and  either have girth 3 or  vertex-stabilizers of order 1 or 2. 

There are  many results about vertex-transitive cubic graphs of girth 4 with two edge orbits, but for larger girth almost nothing is known  about 
the distinguishing costs of such graphs.
We prove that cubic vertex-transitive graphs of girth 5 with two edge orbits  have distinguishing cost 2, and prove  the non-existence of infinite 3-arc-transitive cubic graphs of girth 6.
\end{abstract}

Math. Subj. Class.: {05C15, 05C10, 05C25, 05C63,  03E10.}                       


\section{Introduction}

Cubic vertex-transitive graphs have been one of the most studied objects in algebraic graph theory, the first results dating back to the classic works of Foster \cite{Foster1932, Foster1988} and Tutte \cite{Tutte1947}. 
This topic has remained relevant over time, as evidenced by the works of  Coxeter, Frucht, and Powers \cite{Coxeter1981}, Djokovic and Miller \cite{DjokovicM1980}, Goldschmidt \cite{Goldschmidt1980},  Conder and Lorimer \cite{ConderL1989}, Maru\v si\v c and Scapellato \cite{MarusicS1998},  Glover and Maru\v si\v c \cite{GloverM2007}, Poto\v cnik, Spiga, and Verret \cite{PotocnikSV2015a}, and many others (see also \cite{BarbieriGS2023,  CaoKZ2023, PotocnikT2023} for some very recent works on this topic). 

Meanwhile, the {\em girth} of a graph (shortest cycle length), is a significant characteristic of a graph that appears in numerous well-known problems, results, and formulas. In the class of graphs of bounded girth, one sees that the structure is heavily restricted.

The structure of finite cubic arc-transitive graphs of girth at most 7 and 9 was studied in \cite{FengN2006} and \cite{ConderN2007}, respectively, and those of girth 6 were completely determined in \cite{KutnarM2009a}.
Finite cubic vertex-transitive graphs of girth 3, 4, and 5 were classified by Eiben, Jajcay, and Šparl \cite{EibenJS2019}, while the classification of finite cubic vertex-transitive graphs of girth 6 was done by Potočnik and Vidali \cite{PotocnikV2022}.

Following
Albertson and Collins \cite{AlbertsonC1996}, a vertex coloring of a graph $\Gamma$ is said to be {\em distinguishing} if
the identity is the only automorphism of $\Gamma$ that preserves the coloring. The smallest number of colors of a distinguishing coloring is the {\em distinguishing number} $D(\Gamma)$ of $\Gamma$. Observe that $D(\Gamma)=1$ if and only if $\Gamma$ is asymmetric (has trivial automorphism group), and otherwise $D(\Gamma)\geq 2$.
When $D(\Gamma) = 2$ each of the two colors induces a set of vertices which
is preserved only by the identity automorphism. We call such sets {\em distinguishing}, but the term asymmetrizing was also used by Babai \cite{Babai1977}.
It was proved by H\"uning et al, \cite{HuningIKST2019} that the only cubic vertex-transitive graphs with $D(\Gamma)> 2$ are $K_{3,3}$, $K_4$, the cube, and the Petersen graph.

The cardinality of a smallest distinguishing set of a graph $\Gamma$ is the {\em 2-distinguishing cost}.
It was introduced by Boutin \cite{Boutin2008} in 2008 and denoted $ \rho(\Gamma)$. Clearly $0 < \rho(\Gamma) \leq |V (\Gamma)|/2$.
Although we cannot talk of the $2$-distinguishing cost unless we already know that $D(\Gamma) = 2$, when it is
clear from the context, we refer to $\rho(\Gamma)$ as the distinguishing cost, or simply as the cost,
without adding that $D(\Gamma) = 2$.

Cubic vertex-transitive graphs can naturally be divided into three categories, depending on the number of edge orbits. Those with one edge orbit are edge-transitive and also arc-transitive, those with three edge orbits have trivial vertex-stabilizers, and are the so-called graphical regular representations (GRRs), and the remaining category are the cubic vertex-transitive graphs with two edge-orbits, which are further divided into the class of \emph{rigid graphs}, whose vertex-stabilizers have order 2, and into the class of \emph{flexible graphs}, which have vertex-stabilizers of order at least 4 (see \cite[Corollary 7.2]{ImrichLTW2022}).

It is easy to observe that GRRs have distinguishing cost equal to 1.
In \cite{ImrichLTW2022} it was proved that cubic arc-transitive graphs (different from $K_{3,3}$, $K_4$, the cube, and the Petersen graph, which are not 2-distinguishable) have cost at most 5, unless $\Gamma$ is the infinite $3$-valent tree, which has infinite cost. For cubic vertex-transitive graphs with two edge-orbits, the situation is more complex. 
In \cite{ImrichLTW2022} it was proved that a cubic vertex-transitive graph which is rigid has distinguishing cost 2, unless it is an infinite ladder, a $k$-ladder, or a $k$-M\" obius ladder
with $k > 3$ (which have distinguishing cost equal to 3). The situation for the flexible cubic vertex-transitive graphs heavily depends on the girth. In  \cite{ImrichLTW2022} it was proved that a flexible cubic vertex-transitive graph of girth 3 has distinguishing cost equal to 2, unless it is a truncation of  one of $K_{3,3}$, the Heawood graph or the Tutte-Coxeter graph, in which case it has distinguishing cost equal to 3. For girth 4, a family of cubic vertex-transitive flexible graphs of girth 4 with arbitrarily large distinguishing cost is constructed.

This paper is  concerned with the existence of finite and infinite vertex-transitive cubic graphs of certain girths and their distinguishing costs. It continues the investigations of   \cite{ImrichLTW2022} and also solves several open problems posed there. We prove that cubic vertex-transitive graphs of girth 5 with two edge orbits have distinguishing cost equal to 2. In Theorem~\ref{thm:girth 4 and 5} we classify cubic graphs of girth $g\leq 5$ admitting a consistent cycle of length $g$ and such that every edge of $\Gamma$ is contained in a $g$-cycle. For $g=4$ the only example is $K_{3,3}$ and for $g=5$ the only examples are the Petersen graph and the dodecahedron. In \cite{ImrichLTW2022}, the distinguishing cost of cubic arc-transitive graphs was considered in Section 6.  It is known that all such graphs are $s$-arc regular for $s\le 5$ or is the infinite trivalent tree.  They considered the distinguishing cost for the various values of $s$, and observed that all known examples are finite.  This naturally leads to \cite[Question 6.9]{ImrichLTW2022}, which asks if there are any infinite $3$-arc-transitive graphs of valency $3$ and girth $6$.  We show there are no such graphs in Corollary \ref{final result}, and give the application of this result to distinguishing cost in Corollary \ref{distinguishing application}.  Along the way, we also obtain a new proof of Conder and Nedela's result that the only finite connected cubic $3$-arc-transitive graphs are the Heawood graph, the Pappus graph, and the Desargues graph

\section{Preliminaries}\label{sec:prelim}

Let $X$ be a graph.
A subgroup $G \leq \Aut (X)$ is said to be {\em vertex-transitive}, {\em edge-transitive},
 and {\em arc-transitive} provided it acts transitively on the sets  of vertices, edges  and arcs
of $X$, respectively. In this case the graph $X$ is said to be $G$-{\em vertex-transitive}, $G$-{\em edge-transitive},
 and $G$-{\em arc-transitive}, respectively. In the case where $G = \Aut (X)$, the symbol $G$ is omitted.
An arc-transitive graph is also called {\em symmetric}.
For a positive integer $s$, an {\em $s$-arc} of $X$ is defined as a sequence of vertices $x_0x_1\ldots x_s)$ of $X$ such that $x_i$ is adjacent to $x_{i+1}$ ($i\in \{0,\ldots,s-1\}$) and $x_i\neq x_{i+2}$.
A subgroup $G\leq \Aut (X)$ is said to be {\em $s$-arc-transitive} if it acts transitively on the set
of $s$-arcs of $X$, and it is said to be {\em $s$-regular} if it is $s$-arc-transitive
and the stabilizer of an  $s$-arc in $G$ is trivial.
A graph $X$ is said to be {\em $(G,s)$-arc-transitive} and {\em $(G, s)$-regular}
if $G$ is transitive and regular on the set of $s$-arcs of $X$, respectively.
A $(G, s)$-arc-transitive graph is said to be $(G, s)$-transitive if the graph is not $(G, s + 1)$-arc-transitive.
By Weiss~\cite{Weiss1974b,Weiss1974}, for a pentavalent  $(G, s)$-transitive graph, $s \geq 1$,
the order of the vertex stabilizer $G_v$ in $G$  is a divisor
of $2^{17} \cdot 3^2 \cdot 5$. In addition,
the complete classification of vertex-stabilizers can be deduced from his work, as was recently observed by Guo and Feng~\cite[Theorem 1.1]{GuoF2012}.

\begin{proposition}
\label{pro:feng}
{\rm \cite[Theorem 1.1.]{GuoF2012}}
Let $X$ be a connected pentavalent $(G, s)$-transitive graph for some $G \le Aut(X)$ and $s \ge 1$.
Let $v \in V(X)$. Then $s \le 5$ and one of the following holds:
\begin{enumerate}[(i)]
\itemsep=0pt
\item For $s = 1$, $G_v\cong \ZZ_5, D_{10}$ or $D_{20}$;
\item For $s = 2$, $G_v\cong F_{20}$, $F_{20} \times \ZZ_2$, $A_5$ or $S_5$;
\item For $s = 3$, $G_v\cong F_{20} \times \ZZ_4$, $A_4 \times A_5$, $S_4 \times S_5$ or
	$(A_4 \times A_5) \rtimes \ZZ_2$ with $A_4 \rtimes \ZZ_2 = S_4$ and $A_5 \rtimes \ZZ_2 = S_5$.
\end{enumerate}
\end{proposition}

The following result will be needed for the study of the distinguishing cost of cubic vertex-transitive graphs with girth 5.

\begin{lemma}\label{lem:5-valent stabilizer of uv and their neighbours}
Let $X$ be a connected 5-valent $G$-arc-transitive graph where $G_v$ is solvable, and let $uv$ be an edge. The only element of $G$ that fixes $u,v$ and all of their neighbors in $X$ is the identity.
\end{lemma}
\begin{proof}
Follows from the proof of \cite[Theorem 4.1]{ZhouF2010}.
\end{proof}

\section{Distinguishing cost of cubic vertex-transitive graphs of girth 5}

In this section we will study the distinguishing cost of cubic vertex-transitive graphs of girth 5. Such graphs were classified by Eiben, Jajcay, and \v Sparl in \cite{EibenJS2019}. Before presenting this classification, we need to introduce some definitions.

Let $\Lambda$ be a finite $k$-regular graph and let $D(\Lambda)$ denote the set of its arcs. A {\em vertex-neighborhood labeling} of $\Lambda$ is a function $\rho:D(\Lambda) \mapsto \{1, 2,\ldots, k\}$ such that for each $u \in V(\Lambda)$ the restriction of $\rho$ to the set $\{(u, v):v\in \Lambda(u)\}$ of arcs emanating from $u$ is a bijection. Furthermore, let $Y$ be a graph of order 
$k$ with $V(Y) =\{v_1, v_2, \ldots, v_k\}$. The generalized truncation $T(\Lambda, \rho; Y)$ of $\Lambda$ by $Y$ with respect to $\rho$ is the graph with the vertex set $\{(u, v_i):u \in V(\Lambda), 1 \leq i \leq k\}$ and edge set
$\{ (u, v_i)(u, v_j) | v_iv_j \in  E(Y) \} \cup \{ (u, v_{\rho(u,w)})(w, v_{\rho(w,u)}) | uw \in E(\Lambda) \}.$

\begin{theorem}{\rm \cite[Theorem 6.3]{EibenJS2019}}\label{thm:generalized truncations}
Let $\Gamma$ be a connected cubic $G$-vertex-transitive graph of girth $5$. Then $\Gamma$ is isomorphic to the Petersen graph, the dodecahedron, or there exists a connected $5$-valent $G$-arc-transitive graph $\Lambda$ with the property that the induced action of $G_x$ on the neighbors $\Lambda(x)$ of a vertex $x$ in $\Lambda$ is isomorphic to $C_5$ or $D_{10}$ such that $\Gamma$ is isomorphic to a generalized truncation of $\Lambda$ by the $5$-cycle.
\end{theorem}

\begin{remark}
Based on the communication with the authors of \cite{EibenJS2019}, it was discovered that the last statement of \cite[Theorem 6.3]{EibenJS2019}, saying that $\Aut(\Gamma)\cong \Aut(\Lambda)$ is incorrect, and that the correct statement, which is actually proved in \cite{EibenJS2019} is $\Aut(\Gamma)\leq \Aut(\Lambda)$. The same remark applies to the classifications of cubic vertex-transitive graphs with girth 3 or 4 given in  \cite{EibenJS2019}.
\end{remark}

\begin{corollary}
Let $\Gamma$ be a connected cubic $G$-vertex-transitive graph of girth $5$ that is not isomorphic to the Petersen graph or the dodecahedron and $v\in V(\Gamma)$. Then $|G_v|\in \{1,2,4\}$.
\end{corollary}
\begin{proof}
By Theorem~\ref{thm:generalized truncations} there exists a connected $5$-valent $G$-arc-transitive graph $\Lambda$ with the property that the induced action of $G_x$ on $\Lambda(x)$, $x\in V(\Lambda)$, is isomorphic to $C_5$ or $D_{10}$ such that $\Gamma$ is isomorphic to a generalized truncation of $\Lambda$. Since $G_x$ is not 2-transitive on $\Lambda(x)$ it follows that $G$ does not act transitively on the set of 2-arcs in $\Lambda$. By Proposition~\ref{pro:feng} it follows that $|G_x|\in \{5,10,20\}$ for any vertex $x$ of $\Lambda$. Observe that $|G|=|V(\Lambda)||G_x|$ and $|G|=|V(\Gamma)||G_v|$, where $v$ is a vertex of $\Gamma$. Since $|V(\Gamma)|=5|V(\Lambda)|$ it follows that $|G_v|\in \{1,2,4\}$.
\end{proof}

\begin{theorem}
Let $\Gamma$ be a connected cubic vertex-transitive graph of girth $5$ with two edge orbits. Then the distinguishing cost of $\Gamma$ is 2.
\end{theorem}
\begin{proof}
Let $\Gamma$ be a connected cubic vertex-transitive graph of girth 5 with two edge orbits. Let $G=\Aut(\Gamma)$. Color the edges of $\Gamma$ whose orbits induce a perfect matching with the color red, and other edges with black. Let $v_1u_1$ be an arbitrary red edge, and let $C_1=v_1v_2v_3v_4v_5$ and $C_2=u_1u_2u_3u_4u_5$ be the black cycles of length 5 containing $v_1$ and $u_1$ (see Figure~\ref{fig1}).
First observe that $v_1u_1$ is the only edge between the cycles $C_1$ and $C_2$. This follows from the fact that $\Gamma$ is a generalized triangulation, and from the definition of generalized triangulations.

\begin{figure}[h]
\begin{center}
\includegraphics[scale=0.6]{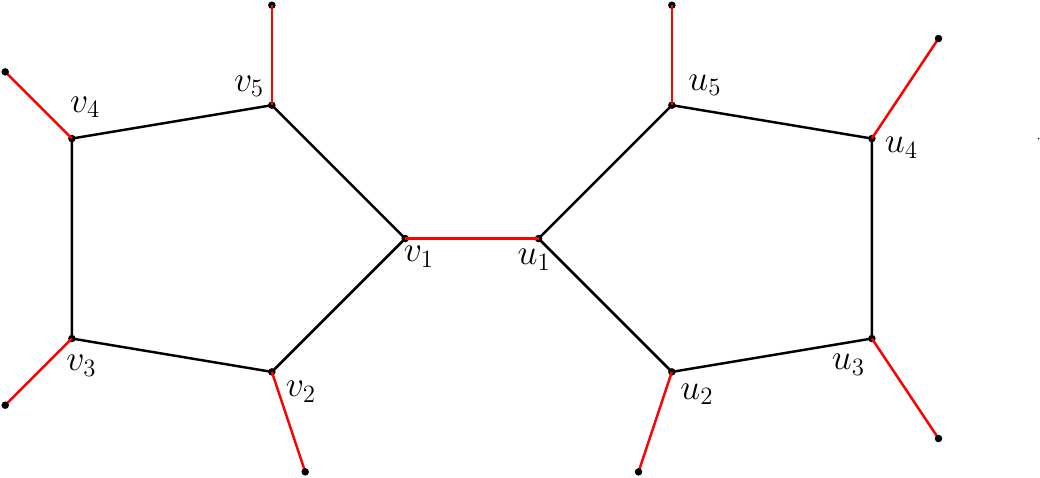}
\caption{Local structure of $\Gamma$ around $u_1v_1$. Black and red edges indicate different edge orbits.}
\label{fig1}
\end{center}
\end{figure}

Let $S=\{v_2,u_3\}$. We claim that the only element of $G$ that fixes the set $S$ is the identity.
Let $g\in \Aut(\Gamma)$ fix the set $S$ setwise.
Observe that $v_2v_1u_1u_2u_3$ is a $BRBB$ (black-red-black-black) path of length $4$ between $v_2$ and $u_3$.
Since $g$ fixes $\{v_2,u_3\}$ setwise, it follows that $g$ must map this $BRBB$ path of length $4$ into another $BRBB$ path of length $4$ from $v_2$ to $u_3$ or from $u_3$ to $v_2$. However, since $v_1u_1$ is the only red edge between $C_1$ and $C_2$ it follows that $v_2v_1u_1u_2u_3$ is the unique $BRBB$ path of length $4$ between $v_2$ and $u_3$, so all points lying on this $4$-path must be fixed, that is, $g$ fixes each of the points $v_2,v_1,u_1,u_2,u_3$. It follows that $g$ fixes all of the points $v_i$ and $u_i$, and their neighbors.

Recall that $\Gamma$ is a generalized triangulation of a $5$-valent arc-transitive graph $\Lambda$. 
Let $\tilde{G}$ be the subgroup of $\Aut(\Lambda)$ induced by the action of $G$. Note that $G\cong \tilde{G}$ and that $\tilde{G}$ acts 1-arc-transitively on $\Lambda$.

Then $g$ corresponds to an automorphism $\tilde{g}$ of $\Lambda$, and it follows that $\tilde{g}$ fixes two adjacent vertices of $\Lambda$ (corresponding to $5$-cycles $C_1$ and $C_2$) as well as all of their neighbors.
By Lemma~\ref{lem:5-valent stabilizer of uv and their neighbours} it follows that the only such automorphism of $\Lambda$ is the identity. We conclude that $g$ is the identity.
\end{proof}

The next natural step would be to consider cubic vertex-transitive graphs which are flexible with girth larger than 5. The classification of cubic vertex-transitive graphs of girth 6 was obtained by Poto\v cnik and Vidali \cite{PotocnikV2022}, so one approach would be to try using this classification to determine the distinguishing costs of those graphs, which we leave as an open question for further work.
\begin{problem}
Determine the distinguishing cost of cubic vertex-transitive graphs of girth 6.       
\end{problem}
Moving to larger girths, the authors are not aware of constructions of cubic flexible vertex-transitive graphs with arbitrary large girth $g$, which we pose as an open problem.

\begin{problem}\label{prob:flexible large girth}
Does there exist a connected flexible cubic vertex-transitive graph of girth $g$ for every positive integer $g\geq 6$?
\end{problem}

A cubic vertex-transitive graph that has two edge-orbits is 
a circular ladder graph or a M\"obius ladder graph, or is obtained from a $4$-valent arc-transitive graph admitting  an arc-transitive cycle decomposition \cite[Corollary 13]{PotocnikSV2013}. By \cite{MiklavicPW2007}, the permutation group induced by the action of the point stabilizer on its neighbors is  $\mathbb{Z}_2\times \mathbb{Z}_2$, $\mathbb{Z}_4$ or $D_4$. If the action of the point stabilizer in a $4$-valent arc-transitive graph is isomorphic to $\mathbb{Z}_2\times \mathbb{Z}_2$ or $\mathbb{Z}_4$, then the obtained cubic vertex-transitive graphs are rigid, and the flexible ones correspond to those obtained from $4$-valent arc-transitive graphs with point stabilizer acting locally on neighbours as $D_4$.
It follows that one approach to answer Problem~\ref{prob:flexible large girth} is to start with a 4-valent arc-transitive graph of large girth with local action of point stabilizer isomorphic to $D_4$ admitting the so-called arc-transitive cycle decomposition, and then use the method explained in \cite{PotocnikSV2013} to construct a cubic flexible graph of large girth.

\section{Consistent cycles in cubic graphs}
In this section we will study consistent cycles in cubic graphs with small girth. 
A cycle in a graph is {\it consistent if the automorphism group of the graph admits a one‐step rotation of this cycle.}
The main motivation for this section is to answer a question posed in \cite[Question 6.9]{ImrichLTW2022}.  We will show in Corollary \ref{final result} that for $s\ge 3$ there are no infinite cubic $s$-arc-transitive graphs of girth $6$.
The authors noticed that similar type results \cite[Proposition 3.4]{GloverM2007} for finite graphs first showed that $s$-arc-transitive graphs of small girth  have consistent cycles. Our proof of Corollary \ref{final result} uses consistent cycles, which were also central to the arguments of  \cite[Proposition 3.4]{GloverM2007} which classifies finite cubic arc-transitive graphs of girth at most five.
We will thus first study cubic graphs $\Gamma$ of small girth with a consistent cycle with the additional condition that every edge of $\Gamma$ is contained in a girth cycle, an obvious condition that edge-transitive, and hence arc-transitive, graphs also possess. 

\begin{theorem}\label{thm:girth 4 and 5}
Let $\Gamma$ be a connected cubic graph (finite or infinite) of girth $g$ with a consistent cycle of length $g$ and such that every edge of $\Gamma$ is contained in a $g$-cycle.  If $g = 4$ then $\Gamma$ is $K_{3,3}$ or the cube, while if $g = 5$, then $\Gamma$ is the Petersen graph or the dodecahedron.
\end{theorem}

\begin{proof}
Let $v_0v_1v_2\ldots v_{g-1}v_0$ be a consistent $g$-cycle in $\Gamma$.  Let $\gamma\in \Aut(\Gamma)$ such that $\gamma(v_i) = v_{i + 1}$ with arithmetic in the subscript done modulo $g$.  As $\Gamma$ is cubic,  there exists $u_i\in V(\Gamma)$ with $u_iv_i\in E(\Gamma)$ and no $u_i$ is any $v_j$, where $i,j\in\Z_g$. We consider various cases depending on the different girths separately.

{\bf Case 1:} $g=4$.

Suppose first that  some $u_i = u_j$.  Note that it cannot be the case that $u_0 = u_1$ or $u_0 = u_3$ as both of those possibilities give a $3$-cycle in $\Gamma$ which has girth $4$.  Then $u_0 = u_2$.  As $\gamma(v_i) = v_{i + 1}$ we have that $u_1 = u_3$.  Also, $u_0\not = u_1$ as $\Gamma$ is cubic.  If $u_0u_1\in E(\Gamma)$, then it is easy to see that $\Gamma \cong K_{3,3}$ and the result follows.  Otherwise, there is $w_0,w_1\in V(\Gamma)$ with $u_0w_0,u_1w_1\in E(\Gamma)$, but $w_0,w_1\not\in\{v_i,u_j:i\in\Z_4,j\in\Z_2\}$.  Note that if $w_0 = w_1$, then $w_0$ is adjacent to some vertex $x$ with $x\not\in\{v_i,u_i,w_0\}$.  
But then $w_0x$ is a bridge of $\Gamma$,  contradicting the assumption that $w_0x$ is contained in a $4$-cycle.  So $w_0\not = w_1$.  Let $D$ be a $4$-cycle in $\Gamma$ that contains the edge $w_1u_1$, and let $\Delta = \Gamma[\{v_i,u_j,w_j:i\in\Z_4,j\in\Z_2\}]$.  As every vertex of $\Delta$ other than $w_0$ and $w_1$ has valency $3$, $D$ must contain a path in $\Delta$ from $u_1$ to $w_0$.  However, the distance in $\Delta$ from $u_1$ to $w_0$ is $4$.  Thus $D$ cannot exist, a contradiction.

Suppose now that the $u_i$ are all distinct, $i\in\Z_4$.  Then $u_0v_0v_1u_1$ is a path on four vertices in $\Gamma$, as is $u_3v_3v_0u_1$.  Any $4$-cycle $D$ in $\Gamma$ that contains $v_0u_0$ contains one of the two previously mentioned paths of length $3$.  So either $u_0v_0v_1u_1u_0$ or $u_3v_3v_0u_0u_3$ is a $4$-cycle in $\Gamma$.  As $\gamma(v_i) = v_{i + 1}$, and $u_i$ is the unique neighbor of $v_i$ not on $v_0v_1v_2v_3v_0$, we see $\gamma(u_i) = u_{i + 1}$.  We conclude in either case that $u_iu_{i+1}\in E(\Gamma)$ for all $i\in\Z_4$.  Then $\Gamma$ is isomorphic to the cube.

{\bf Case 2:} $g=5$.

As $\Gamma$ has girth $5$, the $u_i$, $0\le i\le 4$, are all distinct.  If some $u_iu_j\in E(\Gamma)$, $i\not = j$, then, as $\Gamma$ has girth $5$, it must be that $u_iu_{i + 2}$ or $u_iu_{i+3}\in E(\Gamma)$.  As $\gamma(u_i) = u_{i + 1}$, we see that both of these cases occur, and that $u_iu_{i + 2},u_iu_{i + 3}\in E(\Gamma)$ for every $i\in\Z_5$.  The subgraph induced by the vertices $\{u_i,v_i:i\in\Z_5\}$ is $3$-regular, and so all of $\Gamma$.  By inspection of the edges, this graph is the Petersen graph.

Suppose now that $u_iu_j\not \in E(\Gamma)$ for $i\not = j$.  Then each $u_i$ is adjacent to two additional vertices, say $w_{i,0},w_{i,1}$, $i\in\Z_5$.  It cannot be the case that the $w_{i,0}$ and $w_{i,1}$ are all distinct, as if that were the case, the edges $u_iv_i$ would not be contained in a $5$-cycle.  Additionally, $w_{i,0}$ must then be  in the same  $5$-cycle as some of  $w_{i\pm 1,0}$ or $w_{i\pm 1,1}$ as, again, the edges $u_iv_i$ would not be contained in a $5$-cycle.  This means $\{w_{i,j}:i\in\Z_5,j\in\Z_2\}$ consists of $5$ distinct vertices $w_0,\ldots,w_4$, and $u_iw_i,u_{i+1}w_i\in E(\Gamma)$.  As each $w_i$ has valency $2$ in the subgraph of $\Gamma$ constructed so far, there exist vertices $z_0,\ldots,z_4$, distinct from any $u_i,v_i,w_i$, such that $w_iz_i\in E(\Gamma)$.  Note that if the $z_i$ are not all distinct, then $\Gamma$ either has girth $4$ (if $z_i = z_{i + 1}$) or each $z_i$ has valency at least $4$ (after applying $\gamma$ repetitively).  
The only way in which these edges can be contained in $5$-cycles is if $z_iz_{i +1}\in E(\Gamma)$ for all $i\in\Z_4$.  The subgraph induced by the vertices $\{u_i,v_i,w_i,z_i:i\in\Z_4\}$ is $3$-regular, and so all of $\Gamma$.  By inspection of the edges, this graph is the dodecahedron.
\end{proof}

We now consider the case when $\Gamma$ has girth $6$.  The proof is quite a bit longer, and so is broken into two natural cases.  In both cases, we show that such a graph is either finite, or has restricted `local' structure.  We also show that in the latter case, the graph cannot have the property that a $3$-arc is contained in a $6$-cycle.

\begin{lemma}\label{more than 12 edges}
Let $\Gamma$ be a cubic graph of girth $6$ (finite or infinite) with a consistent girth cycle and with every edge of $\Gamma$ contained in a $6$-cycle.  Assume that the subgraph $\Delta$ of $\Gamma$ induced by a consistent cycle and its neighbors has more than $12$ edges.  Then $\Gamma$ is either the Heawood graph or contains the graph in Figure \ref{subgraph of order 18} as a subgraph.  If $\Gamma$ also has the property that every $3$-arc is contained in a $6$-cycle, then $\Gamma$ is the Heawood graph. 
\end{lemma}

\begin{proof}
Let $C = v_0v_1\ldots v_5v_0$ be a consistent cycle in $\Gamma$.  Let $\gamma\in \Aut(\Gamma)$ such that $\gamma(v_i) = v_{i + 1}$ with arithmetic in the subscript done modulo $6$.  As $\Gamma$ is cubic and $C$ is a $6$-cycle, the vertex $v_0$ must be adjacent to some other vertex $u_0\not = v_i$, $i\in\Z_6$.  Repetitively applying $\gamma$, we see that there are vertices $u_i\in V(\Gamma)$, all distinct from $\{v_i:i\in\Z_6\}$, and $v_iu_i\in E(\Gamma)$ for every $i\in\Z_6$.  This subgraph of $\Gamma$ is shown in Figure \ref{base graph}, and will be called the base graph.

\begin{figure}[ht]
\begin{center}
\includegraphics[scale=0.7]{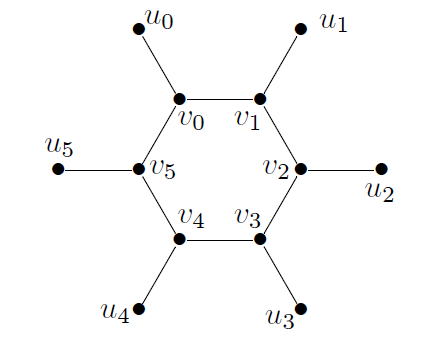}
\caption{The Base Graph.}
\label{base graph}
\end{center}
\end{figure}

Note that the vertex sets of the base graph and $\Delta$ are the same.  As we are assuming $\Delta$ has more than $12$ edges, some of the $u_i$, $i\in\Z_6$, are adjacent to each other.  As $\Gamma$ has girth $6$, the only possibilities are that $u_iu_{i+3}\in E(\Gamma)$ for some $i\in\Z_6$.  Applying powers of $\gamma$ to this edge, we see that $u_iu_{i + 3}\in E(\Gamma)$ for every $i\in\Z_6$.  We have now determined $15$ edges in $\Gamma$ with $v_i$ of valency $3$ and $u_i$ of valence $2$ in the subgraph of $\Gamma$ constructed so far, and each such edge is contained in a $6$-cycle.  Thus each $u_i$ is adjacent to some other vertex of $\Gamma$, and this other vertex is neither a $v_i$ or a $u_i$.  We let $w_0\in V(\Gamma)$ with $u_0w_0\in V(\Gamma)$.  We will consider two cases, depending upon whether or not $w_0$ is adjacent to more than one $u_i$, $i\in\Z_6$.

If $w_0$ is adjacent to more than one $u_i$, then, as $\Gamma$ has girth $6$, $w_0$ cannot be adjacent to $u_1$, $u_3$, or $u_5$. Applying $\gamma$ to the edges and vertices we have identified so far, we see $u_1w_1\in E(\Gamma)$ for some $w_1\in V(\Gamma)$, $v_i\not = w_i\not= u_1$, $i\in\Z_6$, and $w_1\not = w_0$ (as $\Gamma$ has girth $6$).   If $w_0$ is adjacent to $u_2$, then $w_1$ is adjacent to $u_3$.  Applying $\gamma$ to the edges and vertices we have now found, we see that $w_0$ is adjacent to $u_4$ and $w_1$ is adjacent to $u_5$.  Every vertex is now cubic, and so the graph has been constructed.  See the left-hand side of Figure \ref{Heawood}.  We observe that the given graph is the Heawood graph - the right-hand side of Figure \ref{Heawood} with the labeling used here is a usual drawing of the Heawood graph (see \cite[Figure 4.4]{Book}).  The case where $w_0$ is adjacent to $u_4$ is the reflection of the case where $w_0$ is adjacent to $u_2$, and so also results in the Heawood graph.

\begin{figure}[ht]
\begin{center}

\includegraphics[scale=0.6]{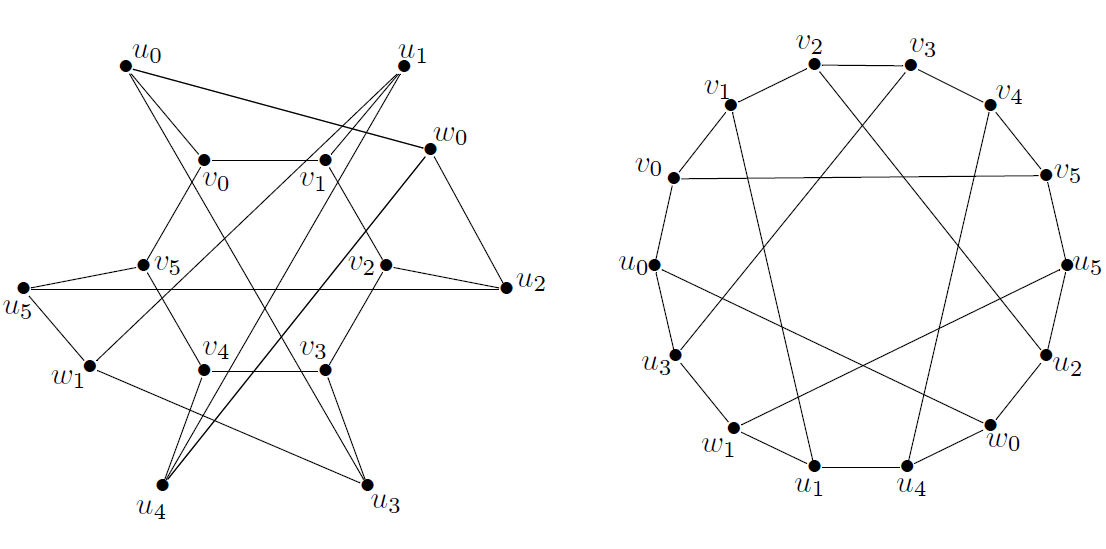}

\caption{The Heawood graph.}

\label{Heawood}
\end{center}
\end{figure}



If $w_0$ is adjacent to exactly one $u_i$, then, repeatedly applying $\gamma$, we see that each $u_i$ is adjacent to a vertex $w_i$ and $w_i\not\in \{v_i,u_i:i\in\Z_6\}$.  Relabeling if necessary, there are vertices $w_0,\ldots,w_5$, none of which are in $\{v_i,u_i:i\in\Z_6\}$, and $u_iw_i\in E(\Gamma)$.  We have now identified $18$ vertices of $\Gamma$.  Each $v_i$ has valency $3$, each $u_i$ has valency $3$, and each $w_i$ has valency $1$.  All edges except for the edges $u_iw_i$, $i\in\Z_6$ are contained in cycles of length $6$.  This is the graph in Figure \ref{subgraph of order 18}, which for convenience we will call $\Omega$.

\begin{figure}[htbp]
\begin{center}
\includegraphics[scale=0.5]{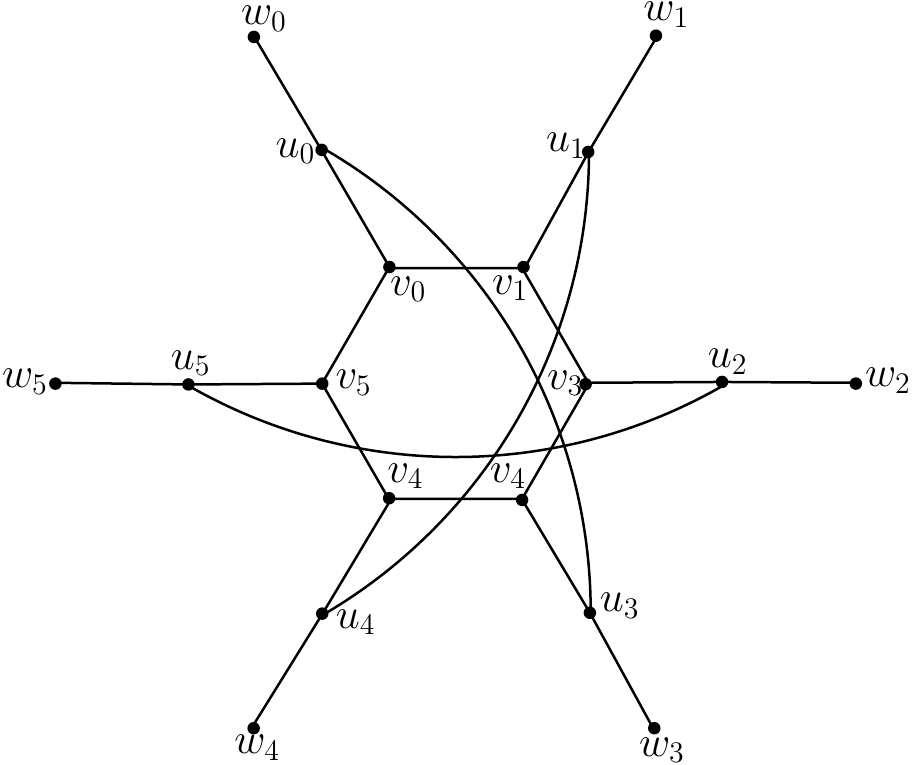}
\caption{A subgraph of $\Gamma$.}\label{subgraph of order 18}
\end{center}
\end{figure}

Consider the $3$-arc $A = u_0v_0v_1u_1$ in $\Gamma$.  In order for $A$ to be contained in a $6$-cycle $C'$ of $\Gamma$, some neighbor of $u_0$ not on $A$ must be adjacent to some neighbor of $u_1$ not on $A$ in $\Gamma$.  The neighbors of $u_0$ and $u_1$ not on $A$ in $\Gamma$ are all contained in $\Omega$, and are the vertices $w_0,w_1, u_3$, and $u_4$.  In $\Omega$, these vertices form an independent set, so $C'$ must contain an edge not in $\Omega$, namely, $w_0w_1$.  Applying $\gamma$ to the edge $w_0w_1$, we see that $w_iw_{i+1}\in E(\Gamma)$, and we obtain the cubic graph $\Lambda$ in Figure \ref{order 18}. 


\begin{figure}[ht]
\begin{center}

\includegraphics[scale=0.6]{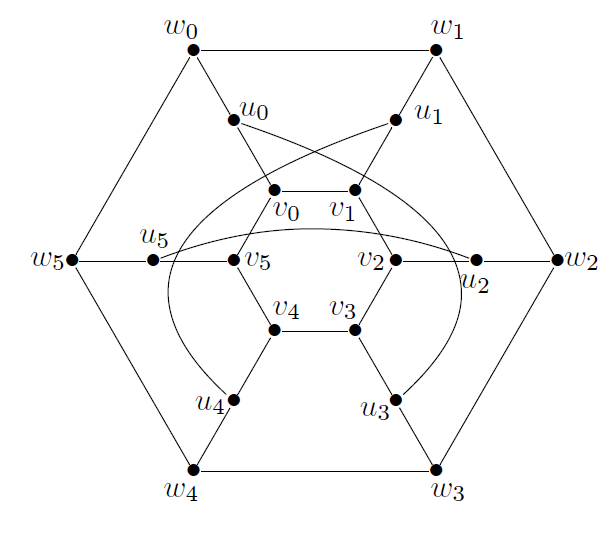}

\caption{A graph of order $18$.}

\label{order 18}
\end{center}
\end{figure}


Consider the $3$-arc $v_5u_5u_2w_2$ in $\Lambda$.  In order to be contained in a cycle in $\Lambda$, the neighbors $v_0,v_4$ of $v_5$ and the neighbors $w_1$ and $w_3$ of $w_2$ in $\Lambda$ must have a common neighbor.  However, $\{v_0,v_4,w_1,w_3\}$ is an independent set in $\Lambda$.  The result follows.
%
\end{proof}

The graph in Figure \ref{order 18} is reminiscent of the Pappus graph, but according to Magma \cite{MAGMA} has automorphism group of order $24$, and was constructed to contain a consistent cycle.  Thus it is an example of a cubic graph of girth $6$ admitting a consistent cycle with every edge being contained in a girth cycle but which is not vertex-transitive.

\begin{lemma}\label{12 edges}
Let $\Gamma$ be a cubic graph of girth $6$ (finite or infinite) that satisfies the following conditions:
\begin{enumerate} 
\item $\Gamma$ has a consistent girth cycle,  
\item every $3$-arc is contained in a $6$-cycle, and
\item the subgraph $\Delta$ of $\Gamma$ induced by a consistent cycle and its neighbors has $12$ edges.  
\end{enumerate} 
Then $\Gamma$ is either the Pappus graph or the Desargues graph. 
\end{lemma}

\begin{proof}
Let $D = v_0v_1v_2v_3v_4v_5v_0$ be a consistent $6$-cycle in $\Gamma$ with $\gamma\in\Aut(\Gamma)$ such that $\gamma(v_i) = v_{i + 1}$.  For each $v_i$ there is $u_i\in \Gamma$ such that $u_iv_i\in E(\Gamma)$, $i\in\Z_6$. Let $U = \{u_i:i\in\Z_6\}$, $V= \{v_i:i\in\Z_6\}$, and $W = \{w\in V(\Gamma):wu_j\in E(\Gamma){\rm\ and\ }w\not = v_k,j,k\in\Z_6\}$.
By hypothesis, the subgraph $\Delta$ of $\Gamma$ induced by $U\cup V$ is the base graph in Figure \ref{base graph}.

By hypothesis, the $3$-arc $u_0v_0v_1v_2$ must be contained in a $6$-cycle $C$. We must also have a neighbor of $u_0$ which is not $v_0$ on $C$, which we will call $w_0$.  As $\Delta$ has $12$ edges we must have that $w_0\not\in V(\Delta)$.

Suppose $\gamma^i(w_0) = w_0$ for some $i\in\Z^+$.  As $u_0w_0\in E(\Gamma)$, $u_iw_0\in E(\Gamma)$, where $u_i = \gamma^i(u_0)\in U$.  As $\Gamma$ has girth $6$, we see that $u_i = u_0,u_2,u_3$, or $u_4$.  Set $w_1 = \gamma(w_0)$.  As $\Gamma$ has girth $6$, $\gamma(v_0) = v_1$, and $\gamma(u_0) = u_1$, $w_1\not = w_0$.  Let $x\in\{2,3,4\}$. 

Let $i\equiv x\ ({\rm mod}\ 6)$.  As $\gamma^i(u_0) = u_x$, we see that $w_0u_x\in E(\Gamma)$ as $\gamma^i(w_0u_0) = w_0u_x\in E(\Gamma)$.  Let $\xi$ be the cycle of $\gamma$ that contains $w_0$.  As $\gamma$ is a $6$-cycle on the vertices in $U$, we see that $\gamma^6$ fixes each vertex of $U$, and the orbit of $\xi^6$ that contains $w_0$ has length $\ell$.  As $\Gamma$ has valency $3$ and $w_0$ is adjacent to a vertex in $U$, we conclude that $\ell \le 3$.  As $u_i$ is adjacent to two vertices that are not in $V$, $\ell\le 2$.  If $w\in W$ and $\gamma(w) = w$, then as $\gamma$ is a $6$-cycle on $U$, we see that $w$ has valency at least $6$, a contradiction.  Hence no element of $W$ is fixed by $\gamma$.  Thus $\xi^6 = 1$ or is a transposition.  This gives that $\xi$ has length $2$, $3$, $4$, $6$, or $12$.

If $\xi$ has length $12$, then $\vert W\vert = 12$, and each $u_j$ is adjacent to exactly two elements of $W$.  Let $w_{j,0}$ and $w_{j,1}$ be the neighbors of $u_j$ in $\Gamma$, $j\in\Z_6$.  The cycle $C$ containing $u_0,v_0,v_1,v_2$ must contain the edge $v_2u_2$ or $v_2v_3$.  Also observe that every neighbor of every vertex in $U\cup V$ has been determined.
If $v_2u_2$ is contained in $C$, then $u_2$ has to be adjacent with one of $w_{0,0}$ or $w_{0,1}$, which is impossible. Similarly, if $C$ contains $v_2v_3$ then $v_3$ has to be adjacent with one of $w_{0,0}$ or $w_{0,1}$, which is also impossible.

If $\xi$ has length $4$, then write $\xi = (w_0,w_1,w_2,w_3)$.  Hence $u_jw_j\in E(\Gamma)$ for $0\le j\le 3$.  Also, $\gamma^4$ has as one of its cycles $(u_0,u_4,u_2)$ and fixes both $w_0$ and $w_2$.  We conclude that both $w_0$ and $w_2$ are adjacent to $u_0,u_2$, and $u_4$.  Thus $w_0u_0w_2u_2w_0$ is a $4$-cycle in $\Gamma$.  So $\gamma$ has no cycle of length $4$ that contains a vertex adjacent to a vertex of $U$.

Suppose $\xi$ has length $2$.  For reasons which will be clear later, we wish to replace the labellings of $w_0$ and $w_1$ with $x_0$ and $x_1$, respectively (and we will reuse the symbols $w_0$ and $w_1$ later).  Then the induced action of $\gamma$ on $\{x_0,x_1\}$ is the transposition $(x_0,x_1)$, and the edges $x_1u_3,u_4x_0,u_5x_1$ are also contained in $E(\Gamma)$.  By hypothesis, the $3$-arc $u_5v_5v_0u_0$ is contained in a $6$-cycle, call it $C'$.  If $u_0x_0\in E(C')$, then $x_0$ and $u_5$ must have a common neighbor.  As $x_0$ has valency $3$ in the subgraph of $\Gamma$ we have now labeled, this common neighbor must be $u_4$ or $u_2$.  If this common neighbor is $u_4$, then $u_2u_4\in E(\Gamma)$, and $\Delta$ has $13$ edges, which is not possible.  If it is $u_2$, then the edge $u_2u_5\in E(\Gamma)$ and again $\Delta$ has $13$ edges.  Thus $u_0x_0\not\in E(C')$, and so $u_0$ is adjacent to some vertex other than $x_1$ and $v_0$, which we will call $w_0$.  Note that $w_0$ cannot be $x_0$ or $x_1$ as we know all of their neighbors in $\Gamma$, and it cannot be in $V(\Delta)$ as $\Delta$ would then have $13$ edges.  Also, $w_0$ and $u_5$ must have a common neighbor, call it $w_5$, in order for $C'$ to exist.  As $w_5$ is a neighbor of $u_5$, $w_5$ is not in $V(\Delta)$ as otherwise $\Delta$ would have $13$ edges.  Thus $w_5\not \in U\cup V$, and it is not $x_0$ or $x_1$ as we know their neighbors.  So $w_5$ is a vertex that we had not previously labeled.

Let $w_1 = \gamma(w_0)$.  As $\gamma$ cannot fix a vertex that is a neighbor of a vertex that is in $U$, $w_1\not = w_0$.  As $\gamma$ is a permutation, $w_1\not\in U\cup V\cup\{x_0,x_1\}$.  Also, we see that $w_1\not = w_5$ as otherwise $\gamma(u_0w_0) = u_1w_5$ and $u_5w_5u_1x_1u_5$ is a $4$-cycle in $\Gamma$.  Thus $w_1$ is not a vertex that we have previously labeled.  The $3$-arc $u_0v_0v_1u_1$ is contained in a $6$-cycle in $\Gamma$, so some neighbor of $u_0$ is adjacent to some neighbor of $u_1$.  As we know all the neighbors of $x_0,v_0,v_1$ and $x_1$, the only possibility is that $w_0w_1\in E(\Gamma)$.  Let $w_2 = \gamma(w_1)$.  We know by previous arguments that $w_2\not\in U\cup V\cup\{x_0,x_1\}$.  Also, $w_2\not = w_1$, and $w_2$ is a neighbor of $u_2$.  As we already know the neighbors of $w_0$, $w_2\not = w_0$.  Finally, if $w_2 = w_5$, the edge $\gamma(w_0w_1) = w_1w_5$ implies that the $3$-cycle $w_5w_0w_1w_5$ is in $\Gamma$.  So $w_2$ is a vertex we have not previously labeled, and $w_1w_2\in E(\Gamma)$.  Using previously seen arguments, the vertex $w_3 = \gamma(w_2)$ is either a vertex that we have not previously labeled, or it is $w_5$ and $\Gamma$ contains the $4$-cycle $w_5w_0w_1w_2$.  Thus $w_3$ is not a vertex we have previously labeled and $w_2w_3\in E(\Gamma)$.  An analogous argument shows that $\gamma(w_3) = w_4$ is a vertex that we have not previously labeled and $w_3w_4\in E(\Gamma)$.  Finally, $\gamma(w_4)$ is a neighbor of $u_5$ in $\Gamma$, and these are $w_5,x_1$ and $v_5$.  As we know all of the neighbors of $x_1$ and $v_5$, we see that $\gamma(w_4) = w_5$ and so $\Gamma$ is determined as every vertex has valency $3$.  It is the graph on the left side of Figure \ref{Desaurges graph}, and on the right side is a standard drawing of the Desargues graph.  We see that $\Gamma$ is isomorphic to the Desaurges graph.  We may now assume that $\gamma$ has no cycle of length $2$ that is adjacent to a vertex of $U$.  

\begin{figure}[ht]
\begin{center}
\includegraphics[scale=0.6]{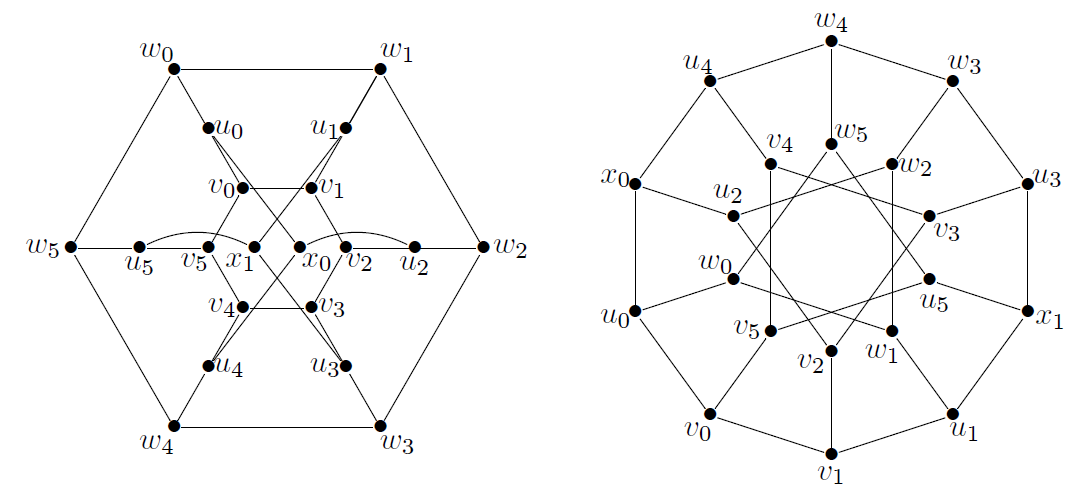}

\caption{The Desargues graph.}

\label{Desaurges graph}
\end{center}
\end{figure}

Suppose $\xi$ has length $3$.  In the cycle decomposition of $\gamma$, we now have that the cycle that contains $w_0$ is a $3$-cycle, and so there is $w_0\not = w_2\not = w_1$ that is a neighbor of $u_2$ in $\Gamma$ that is not $v_2$.  Hence $w_1u_4,w_2u_5\in E(\Gamma)$.  Thus the induced action of $\gamma$ on $U\cup V\cup\{w_0,w_1,w_2\}$ is $$(v_0,v_1,v_2,v_3,v_4,v_5)(u_0,u_1,u_2,u_3,u_4,u_5)(w_0,w_1,w_2).$$
The $3$-arc $u_0v_0v_1u_1$ is contained in a $6$-cycle in $\Gamma$.  Observe that this $6$-cycle cannot be $w_0u_0v_0v_1u_1w_1w_0$ as then $w_0w_1,w_1w_2$, and $w_2w_0\in E(\Gamma)$ and $w_0$ has valency $4$.  Thus $u_0$ is adjacent to some vertex $w_0'$ that is not in $U\cup V\cup\{w_0,w_1,w_2\}$.  Set $w_1' = \gamma(w_0')$.  As $\gamma$ is a permutation, $w_1'$ is a vertex that we have not previously labeled.  By arguments above, we may assume that the cycle $\xi'$ of $\gamma$ that contains $w_0'$ is not a transposition, and so $\gamma(w_1') = w_2'$ and $w_2'$ is not a vertex we have previously labeled.  If $\xi'$ is a $3$-cycle and $\gamma(w_2') = w_0'$, then $\gamma(u_2w_2') = w_0'u_3\in E(\Gamma)$ and $u_3w_0u_0w_0'u_3$ is a a $4$-cycle in $\gamma$.  So $\gamma(w_2') = w_3'$, a vertex we have not previously labeled.  By arguments above, $\xi'$, permuting a vertex adjacent to a vertex in $U$, must have length $6$. 

Now observe that in $\gamma^3$, the vertices $w_0,w_1$, and $w_2$ are fixed, while $\gamma^3$ is a product of three transpositions on $W' = \{w_i':i\in\Z_6\}$.  We conclude that if there is some edge between a vertex in $\{w_0,w_1,w_2\}$ and a vertex of $W'$, then there are two vertices of $W'$ adjacent to a vertex in $\{w_0,w_1,w_2\}$.  This is not possible, as each $w_i$ is adjacent to two vertices in $U$.  Hence no vertex of $\{w_0,w_1,w_2\}$ is adjacent to any vertex of $W'$.  By hypothesis, the $3$-arc $w_0u_0v_0v_1$ is contained in a $6$-cycle $C'$.  Then either $v_1u_1\in V(C')$ or $v_1v_2\in E(C')$.  If $v_1u_1\in E(C')$, then $w_0$ and $u_1$ must have a common neighbor, but as there are no edges in $\Gamma$ between $\{w_0,w_1,w_2\}$ and $W'$, there is no such edge.  Similarly, if $v_1v_2\in E(C')$, then $w_0$ and $v_2$ must have a common neighbor, which they also do not.  Thus $\xi$ cannot have length $3$.  The only remaining possibility is that every cycle of $\gamma$ that contains a vertex adjacent to a vertex of $u$ has length $6$.

If $\gamma$ has two orbits of size $6$ on vertices of $W$, then we see that the $3$-arc $u_0v_0v_1v_2$ does not lie on $6$-cycle in the same way as when $\xi$ had length 12. Hence we conclude that $W$ contains $6$ vertices, and that $\gamma$ induces a $6$-cycle on these $6$ vertices.  Let $\xi=(w_0,w_1,w_2,w_3,w_4,w_5)$ where $u_iw_i$ is an edge in $\Gamma$. 
It is now easy to see that the only way that the $3$-arc $u_0v_0v_1v_2$ lies on $6$-cycle is to have that $w_0u_2$ or $w_2u_0$ is an edge of $\Gamma$. 
It follows that $w_iu_{i+2}$ ($i\in \mathbb{Z}_6$) are edges of $\Gamma$ or $w_iu_{i-2}$ ($i\in \mathbb{Z}_6$) are edges of $\Gamma$. As the second option can be obtained from the first by reflection (replacing index $i$ in each vertex with $-i$), without loss of generality we will assume that  $w_iu_{i+2}$ ($i\in \mathbb{Z}_6$) are edges of $\Gamma$.

\begin{figure}[ht]
\begin{center}
\includegraphics[scale=0.6]{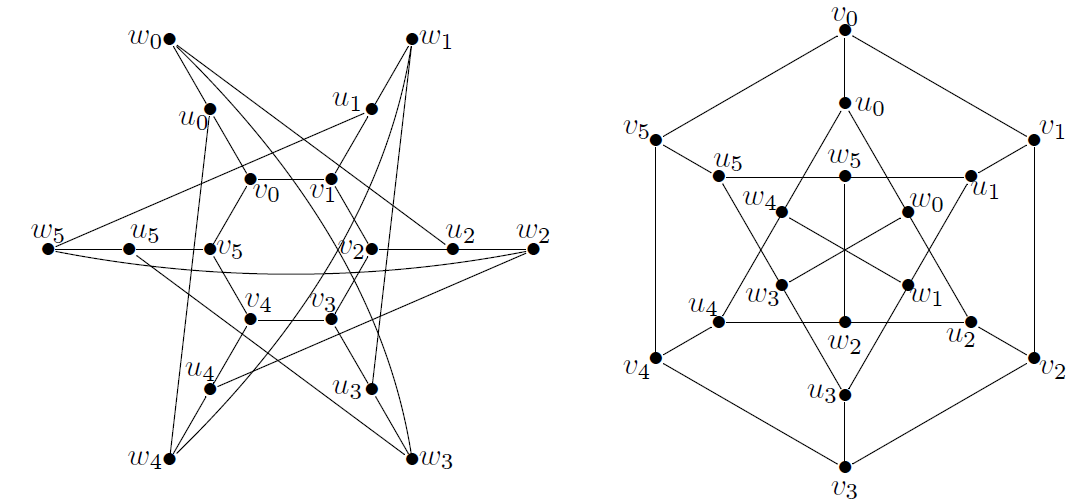}

\caption{The Pappus graph.}

\label{Pappus graph}
\end{center}

\end{figure}

Since the  $3$-arc $u_0v_0v_1u_1$ lies on a $6$-cycle, it follows that there is an edge between some of the neighbors of $u_0$ and $u_1$ different from $v_0$ and $v_1$. It follows that there is at least one edge between vertices in $W$. Since $\gamma$ induces a cycle of length $6$ on $W$, applying $\gamma$, it follows that each vertex of $W$ has a neighbor in $W$. It is now easy to see that the only possibility is that $w_iw_{i+3}$ is an edge of $\Gamma$, as otherwise each vertex of $W$ would have two neighbors in $W$ and two neighbors in $U$, contradicting the assumption that $\Gamma$ is cubic.
Since $\Gamma$ is connected, it follows that $\Gamma$ is of order 18, and it can be seen that it is isomorphic to the Pappus graph (see left hand side of Figure \ref{Pappus graph}).
This concludes the proof.
\end{proof}

The next result is a combination of Lemmas \ref{more than 12 edges} and \ref{12 edges}

\begin{theorem}\label{2.5 and 2.6}
Let $\Gamma$ be a cubic graph of girth $6$ (finite or infinite) that contains a consistent girth cycle and such that every $3$-arc of $\Gamma$ is contained in a $6$-cycle.  Then $\Gamma$ is the Heawood graph, the Pappus graph, or the Desaurges graph.
\end{theorem}

We next turn our attention to which cubic $s$-arc-transitive graphs contain consistent cycles.

\begin{lemma}\label{4-arc transitive}
Let $s\geq 3$. Every cubic $s$-arc-transitive graph $\Gamma$ of girth $s+2$ contains a consistent cycle.
\end{lemma}

\begin{proof}
Let $\Gamma$ be an $s$-arc-transitive graph of girth $s+2$, and let $A = v_0,v_1,\dots,v_s$ be an $s$-arc in $\Gamma$.  As $\Gamma$ is $s$-arc-transitive, has girth $s+2$, and an $(s+2)$-cycle contains an $s$-arc, we see that $A$ is contained in an $(s+2)$-cycle $C = v_0v_1v_2\ldots v_{s+1}v_0$.  Then there exists $\gamma\in\Aut(\Gamma)$ such that $\gamma(A) = v_1,v_2,\ldots,v_s,v_{s+1}$.  Hence $\gamma(v_i) = v_{i + 1}$ for $0\le i\le s$.  If $\gamma(v_{s+1}) = v_0$, then $C$ is a consistent cycle in $\Gamma$.  Otherwise, $\gamma(v_{s+1}) = x\not\in V(C)$.  Then $\gamma(v_{s+1}v_0) = xv_1$, and $v_{s+1}xv_1v_0v_{s+1}$ is a $4$-cycle in $\Gamma$, a contradiction with the girth of $\Gamma$ being equal to $s+2\geq 5$.  Hence $\Gamma$ contains a consistent cycle.
\end{proof}

\begin{lemma}\label{consistent cycle}
Every $3$-arc-transitive graph of girth $6$ has a consistent $6$-cycle.
\end{lemma}

\begin{proof}
Suppose that $\Gamma$ is a cubic $3$-arc-transitive graph of girth $6$ without a consistent $6$-cycle.

{\bf Claim:} Given any $6$-cycle of $\Gamma$, any two antipodal points of the $6$-cycle are endpoints of a path of length $3$ sharing no edge with the $6$-cycle.

Let $C_0=v_0v_1v_2v_3v_4v_5v_0$ be a cycle of length 6 in $\Gamma$.  As $\Gamma$ is $3$-arc-transitive, there exists an automorphism $\alpha$ mapping the $3$-arc $v_0v_1v_2v_3$ to $v_1v_2v_3v_4$.  If $\alpha(v_4)=v_5$, then either $\alpha(v_5) = v_0$ and $C_0$ is a consistent $6$-cycle in $\Gamma$, or $\alpha(v_5) = z\not\in V(C_0)$.  But then $\alpha(v_4v_5v_0) = v_5zv_1$ and $v_5zv_1v_0v_5$ is a $4$-cycle in $\Gamma$, a contradiction showing that  $\alpha(v_4)\not = v_5$.

Suppose now that $\alpha(v_4)=x\neq v_5$. Then $\alpha(\{v_3,v_4\}) = \{v_4,x\}\in E(\Gamma)$. Let $w=\alpha(v_5)$.   Note that as $\Gamma$ has girth $6$, $w\not\in V(C_0)$.  Then $\alpha(\{v_5, v_0\}) = \{w,v_1\}\in E(\Gamma)$ and $v_4xwv_1$ is the claimed path of length $3$, concluding the proof of the claim.

Denote by $u_i$ the neighbour of $v_i$ outside of $C_0$, for $i\in \Z_6$, and so we have the base graph in Figure \ref{base graph} as a subgraph of $\Gamma$.  By the above claim, it follows that $u_0u_3$, $u_1u_4$ and $u_2u_5$ are edges in $\Gamma$. Consider now the $6$-cycle $C_1=v_0v_1v_2u_2u_5v_5v_0$, and observe that $v_0,u_2$ is a pair of antipodal points of $C_1$.  By the Claim there exists a path of length $3$ between them sharing no edge with $C_1$.  It is clear that $u_0$ must be contained in this path, as it is the only neighbour of $v_0$ outside of $C_1$.   Let $x$ be the remaining vertex of this path.   Since $\Gamma$ is cubic of girth $6$ it follows that $x$ is different from all $v_i$ and $u_i$. Hence $x$ is adjacent with $u_0$ and $u_2$.

Consider now the $6$-cycle $C_2=v_0v_1u_1u_4v_4v_5v_0$ and observe that $v_0,u_4$ is a pair of antipodal points of $C_2$. Hence there is a path of length $3$ between $v_0$ and $u_4$ sharing no edge with $C_2$.  Observe that $v_0u_0$ is one edge of this path. Also, $u_3$ cannot be contained on this path, since otherwise we would get the $4$-cycle $v_3v_4u_4u_3v_3$.  It follows that $v_0u_0xu_4$ must be the path, that is $x$ is adjacent with $u_0,u_2$ and $u_4$.

Considering $6$-cycles $v_1v_2v_3u_3u_0v_0v_1$ and $v_3v_4v_5u_5u_2v_2v_3$, it follows that there is another vertex $y$ adjacent with $u_1,u_3$ and $u_5$.  Since the component $D$ of $\Gamma$ that contains $v_1$ is connected and cubic, it follows that $D$ is of order $14$.  From the construction of $D$ it follows that $D$ is the graph on the left hand side of Figure \ref{Heawood}, and so is isomorphic to the Heawood graph.  But the Heawood graph is $4$-arc-regular, and, as is well known has a consistent $6$-cycle (this also follows by Lemma \ref{4-arc transitive}).
\end{proof}

We now prove the main result of this section, and show there are no infinite $s$-arc-transitive graphs of girth $6$ for $s\ge 3$.

\begin{corollary}\label{final result}
Let $s\ge 3$.  There are no infinite cubic $s$-arc-transitive graphs of girth $6$.
\end{corollary}

\begin{proof}
As an $s$-arc-transitive graph is $3$-arc-transitive, it suffices to show that there are no infinite cubic $3$-arc-transitive graphs of girth $6$.  If such a graph $\Gamma$ exists, it has a consistent cycle $C = v_0v_1v_2v_3v_4v_5v_0$ by Lemma \ref{consistent cycle}.  Let $xy\in E(\Gamma)$.  As $\Gamma$ is cubic, $xy$ is contained in some $3$-arc $A$ of $\Gamma$.  As $\Gamma$ is $3$-arc-transitive, there exists $\delta\in\Aut(\Gamma)$ such that $\delta(v_0v_1v_2v_3) = A$.  Then $xy$ is contained in the $6$-cycle $\delta(C)$.  We conclude that every edge of $\Gamma$ is contained in a girth cycle of $\Gamma$.  The result now follows by Theorem \ref{2.5 and 2.6}.
\end{proof}

\begin{corollary}\label{arc-transitive app}
Let $s\ge 3$.  The only connected cubic $s$-arc-transitive graphs of girth $6$ are the Heawood graph, the Desaurgues graph, and the Pappus graph. 
\end{corollary}

We now give applications of the above results for the distinguishing cost for cubic arc-transitive graphs of girth $6$. Our first result improves \cite[Lemma 6.8]{ImrichLTW2022}.

\begin{corollary}\label{distinguishing application}
Let $\Gamma$ be a connected cubic arc-transitive graph of girth $6$.  Then $\Gamma$ is at most $4$-arc-regular.  Also,
\begin{enumerate}
\item if $\Gamma$ is $1$-arc-regular, then $\rho(\Gamma) = 2$,
\item if $\Gamma$ is $2$-arc-regular, then $\rho(\Gamma)\le 3$,
\item if $\Gamma$ is $3$-arc-regular, then $\Gamma$ is the Desaurges graph or the Pappus graph, and $\rho(\Gamma) = 3$,
\item if $\Gamma$ is $4$-arc-regular, then $\Gamma$ is the Heawood graph and $\rho(\Gamma) = 5$.
\end{enumerate}
\end{corollary}

\begin{proof}
All of the information in the result for $s = 1$ and $2$ and comes from \cite[Lemma 6.8]{ImrichLTW2022}.  Except for $\rho(\Gamma)$ in (3) and (4), the result follows by Corollary \ref{arc-transitive app}.  The distinguishing cost of the Desaurgue, Pappus, and Heawood graphs are easily verified by Magma.
\end{proof}

The next result lowers the upper bound on the distinguishing cost given in \cite[Theorem 6.1]{ImrichLTW2022} from $5$ to $4$ by additionally excluding the Heawood graph.

\begin{corollary}
Let $\Gamma$ be a cubic arc-transitive graph that is not $K_4$, $K_{3,3}$, the cube, the Petersen graph, or the Heawood graph.  If $\Gamma$ has finite girth, then $\rho(\Gamma)\le 4$, and otherwise it is the infinite cubic tree $T_3$, which has finite distinguishing cost and distinguishing density $0$ (see \cite[Section 3]{ImrichLTW2022} for the definition of distinguishing density).
\end{corollary}

\begin{proof}
If $\Gamma$ does not have finite girth, then the result follows by \cite[Theorem 6.1]{ImrichLTW2022}.  If $\Gamma$ has finite girth $g$, then if $g\ge 7$ we have $\rho(\Gamma)\le 4$ by \cite[Lemma 6.7]{ImrichLTW2022}.  If $g\le 5$, then by \cite[Theorem 6.2]{ImrichLTW2022} all such graphs have been excluded.  If $g = 6$, then by Corollary \ref{distinguishing application} the only such graph with $\rho(\Gamma) > 4$ is the Heawood graph, which has also been excluded.
\end{proof}

\noindent Acknowledgment.  
The authors are indebted to an anonymous referee for many suggestions which improved the clarity of the paper. The work of Ted Dobson is supported by the Slovenian Research and Innovation Agency grants: P1-0285, J1-50000, J1-4008, and J1-3003. The work of Ademir Hujdurovi\' c is supported by the Slovenian Research and Innovation Agency grants P1-0404,
N1-0428,
J1-50000,
J1-4084, 
N1-0353,
N1-0391, and
J1-60012.
This work was supported by the OEAD: SI 1312020.

\bibliography{References}{}
\bibliographystyle{plain}

 \end{document}